\documentclass{amsart}

\usepackage{latexsym}
\usepackage{amssymb}
\usepackage{amsmath}
\usepackage{amsthm}
\usepackage{amsfonts}
\usepackage{color}
\usepackage{esint}

\usepackage{pictexwd,dcpic}

\usepackage{graphicx}

\usepackage{psfrag}
\usepackage{hyperref}
\usepackage{comment}
\usepackage{multirow} 

\newtheorem{thm}{Theorem}[section]
\newtheorem{cor}[thm]{Corollary}
\newtheorem{lem}[thm]{Lemma}
\newtheorem{prop}[thm]{Proposition}

\theoremstyle{definition}
\newtheorem{defn}[thm]{Definition}

\newtheorem{quest}{Question}[section]

\newtheorem*{ack}{Acknowledgements}

\numberwithin{equation}{section}

\newcommand{\scal}[1]{\langle #1 \rangle}

\newcommand{\RR}{\mathbb{R}}

\newcommand{\vol}{\textrm{vol}}

\newcommand{\sphere}{\mathrm{\mathbb{S}}}

\newcommand{\fol}{\mathcal{F}}

\begin{document}



\title[]{Algebraic nature of singular Riemannian foliations in spheres}


\author[M. Radeschi]{Marco Radeschi}
\address[Radeschi]{Mathematisches Institut, WWU M\"unster, Germany.}
\email{mrade\_02@uni-muenster.de}
\thanks{}

\author[A. Lytchak]{Alexander Lytchak}
\address[Lytchak]{Mathematisches Institut, Cologne, Germany.}
\email{alytchak@math.uni-koeln.de}
\thanks{}

\date{\today}




\begin{abstract}
 We prove that singular Riemannian foliations in Euclidean spheres can be defined by polynomial equations.
\end{abstract}

\maketitle

\bigskip

\section{Introduction}
 Isoparametric hypersurfaces in Euclidean spheres have been studied by Cartan in the thirties (cf. \cite{Car}) and then forgotten for a long period of time.
Such hypersurfaces are natural and very interesting generalizations of (orbits of)  isometric cohomogeneity one actions on spheres.
A major step towards the understanding of isoparametric hypersurfaces has been done
by M\"unzner in \cite{Mun1}, \cite{Mun2}.  He proved a finiteness result controlling the topology of the hypersurfaces and an algebraicity result building a bridge between geometry and algebra: any isoparametric hypersurface is given as the zero set of  a polynomial equation.    Starting from these results  essentially all isoparametric hypersurfaces have been classified by combining  deep
topological, geometric and algebraic insights \cite{Ab}, \cite{Im}, \cite{St}, \cite{Cec}, \cite{Dorf}.

In the same way isoparametric hypersurfaces generalize isometric cohomogeneity one actions, singular Riemannian foliations  generalize
(orbit decompositions of) arbitrary  isometric actions on spheres. Besides
the intrinsic interest in such objects, related to the study of Euclidean  submanifolds with special properties, singular Riemannian foliations in round spheres
describe local structure of singular Riemannian foliations in arbitrary Riemannian manifolds (cf. \cite{Mol}). Thus the understanding of singular Riemannian foliations in spheres is of major importance in the theory.  Molino, not being aware
 of the existence of non-homogeneous isoparametric foliations (cf. \cite{FKM}), has conjectured that all singular Riemannian foliations in Euclidean spheres are homogeneous. However, there is in fact a vast class of non-homogeneous examples  (cf. \cite{Rad}).  Despite this, all
singular Riemannian foliations with closed leaves are of algebraic origin as our main theorem shows:

\begin{thm}\label{T:algebraic}
Let $(\sphere^{n},\fol)$ be a singular Riemannian foliation with closed leaves. Then there exists a polynomial map $\rho=(\rho_1,\ldots \rho_k):\RR^{n+1}\to \RR^k$ such that any leaf of $\fol$ coinicides with some fiber of $\rho$. The induced map $\sphere^{n}/\fol\to \RR^k$ is a homeomorphism   onto the image.
\end{thm}

Our result  identifies the quotient space $\sphere ^n /\mathcal F$ as a semi-algebraic set and provides a related finitely generated algebra of $\mathcal F$-invariant polynomials (Proposition
\ref{hilbert} below). This result opens the way to algebraic methods in the theory of singular Riemannian foliations. Applications of this approach will be discussed in a forthcoming
paper.

As a direct consequence of  Theorem \ref{T:algebraic}  we deduce:

\begin{cor} \label{oneleaf}
Let $(\sphere^{n},\fol)$ be a singular Riemannian foliation with closed leaves. Then  any leaf of $\mathcal F$ is a real algebraic subvariety of the Euclidean space $\RR ^{n+1}$.
\end{cor}

We would like to mention a recent result of a similar spirit and origin.   For a submanifold $L\subset \sphere ^n$ being a leaf of a singular Riemannian foliation  imposes a severe restriction on the set of focal vectors in the normal bundle $\nu L$ of any leaf $L$,
\cite{AT}, \cite{AR15}.  A slightly related restriction on the set of focal vectors of a submanifold is imposed by the assumption of the \emph{tautness} of the submanifold, cf.\cite{Chi}, \cite{CT}, \cite{Wie}.
All isoparametric hypersurfaces are taut (cf. \cite{PT}) and thus the recent result of Q.Chi  that all taut submanifolds of the Euclidean space are algebraic \cite{Chi}, can be seen as a different generalization of M\"unzner's theorem.

A fundamental tool in the proof of Theorem \ref{T:algebraic} is the control of the \emph{averaging operator}, a replacement of the averaging with respect to the Haar measure:

\begin{defn}
Let $(\sphere^n,\fol)$ be a singular Riemannian foliation. For  $f \in L^2 (\sphere ^n )$  the \emph{average of $f$} with respect to $\fol$ is the function $[f]
\in L^2 (\sphere ^n)$ defined at almost every $p\in \sphere ^n$ by
\begin{equation} \label{main}
[f](p)=\fint_{L_p}f  = \frac 1  {vol (L_p) } \int_{L_p} f
\end{equation}
where the integral is taken with respect to  the induced  Riemannian volume   on the leaf $L_p$ through $p$.
\end{defn}

The operator $f\to [f]$ is the orthogonal  projection from $L^2 (\sphere ^n)$
onto the closed linear subspace $L^2 (\sphere ^n, \mathcal F)$ of all representatives of square integrable \emph{basic} function. Recall, that a function is called \emph{basic} (with respect to $\mathcal F$) if it is constant on any leaf of $\mathcal F$.   Much less obvious are the smoothness properties of the average function near singular leaves.  A very closely related problem has been solved in \cite{PR} for the averaging operator of a regular Riemannian foliation with non-closed leaves.
In our case the smoothness is preserved, too:

\begin{thm}\label{smooth average}
Let $(\sphere ^n ,\fol)$ be a singular Riemannian foliation with closed leaves.
 If $f\in  L^2 (\sphere ^n)$  has a smooth, respectively polynomial representative then so does  the  averaged function $[f]$.
\end{thm}

This result implies that the ring of basic polynomials is finitely generated, Proposition \ref{hilbert}, and  provides enough basic polynomials to deduce Theorem \ref{T:algebraic}.

The averaging operator is defined for any  Riemannian manifold $M$ and singular Riemannian foliation $\mathcal F$ with compact leaves.  If the singular foliation
$\mathcal F$ is given by leaf closures of some regular Riemannian foliation $\mathcal G$
on $M$, then our averaging operator with respect to $\mathcal F$ coincides with the averaging operator with respect to $\mathcal G$, as defined in
\cite{PR}.
Thus, Theorem \ref{smooth average}, Theorem  \ref{C:smooth} and \cite{PR} give rise to the hope that the the following question has an affirmative answer:

\begin{quest}
Let $\mathcal F$ be a singular Riemannian foliation with compact leaves on a complete Riemannian manifold $M$.  Does the averaging operator $f\to [f]$ send smooth functions to smooth functions?
\end{quest}

The proof would require a deeper understanding of the structures of the singularities of a singular Riemannian foliation, in particular the behavior of the mean curvature vectors of regular leaves in a small neighborhood of singular leaves. No problems arise if the mean curvature  field is \emph{basic} in the regular part of $\mathcal F$, a very well known condition in the analysis of Riemannian foliations (cf.  \cite{PR} and the literature therein).
In this case the answer to the question above is indeed affirmative (Theorem \ref{C:smooth}).

\begin{ack}
The authors would like to thank  Miguel Dom\'inguez-V\'azquez, Ken Richardson and Wolfgang Ziller   for helpful comments on a previous version of this paper.
\end{ack}

\section{Preliminaries}
Let $M$ always denote a  Riemannian manifold and   let
$\mathcal F$  always denote a  singular Riemannian foliation on $M$ with compact leaves, i.e.,  a decomposition of $M$ as a disjoint union of compact smooth submanifolds $L_p$, called the leaves of $\mathcal F$, such that  the  leaves are equidistant,  and such that smooth vector fields everywhere tangent to the leaves span all tangent spaces to the leaves.  We refer the reader to \cite{Mol}, \cite{ABT}  and the literature therein for  introductions to the subject.  Note that the assumption that all leaves are compact makes further usual assumptions on $M$ like compactness or  completeness irrelevant.

The  manifold $M$ decomposes as a locally finite union of strata, which are smooth  submanifolds of $M$. There is exactly one open and dense stratum, the principal stratum of $M$, denoted by $M_0$.  The restriction of $\mathcal F$ to $M_0$ is given by a Riemannian submersion with compact fibers $\pi:M_0\to B_0$ onto some Riemannian manifold $B_0$.   For a point $p\in M$ we denote by $H(p)$ the mean curvature vector of the leaf $L_p$ through $p$.  By $\kappa $ we denote the dual $1$-form
$\kappa (v):=\scal{v,H}$.  Note that $\kappa$ is a smooth form on $M_0$. But be aware, that $\kappa$ is definitely non-smooth at singular leaves (indeed, $\|H\|^2$   explodes quadratically
as one approaches a singular point, \cite[Prop. 4.3]{AR}).

 We say that $\mathcal F$ has \emph{basic mean curvature} if the form $\kappa$ is a basic $1$-form on the regular part $M_0$, hence if  the smooth horizontal vector field  $H$ on $M_0$ is the horizontal lift of a vector field on $B_0$.

The union $M_1$ of all strata of codimension at most $1$ consists only  of leaves of maximal dimension. The restriction of $\fol$ to $M_1$ is thus a regular Riemannian foliation, and either $M_1=M_0$ (which is always the case, if $M$ is simply connected), or the restriction of $\mathcal F$ to $M_1$ is not transversally oriented.
In the former case, $M_1$ has a double cover $M_1'$ such that the lift $\mathcal F$ to $M_1'$  has only principal leaves.

Any singular Riemannian foliation given by an isometric group action has basic mean curvature. For a general  manifold $M$ and a general inhomogeneous foliation this condition may not hold.  However, in the case $M=\RR ^{n+1}$ or $M=\sphere ^n$ any singular Riemannian foliation has basic mean curvature, since in this
case  the distance to the focal points determines the eigenvalues of the second fundamental form, see \cite[Cor. 4.6]{AR15}.

If $\mathcal F$ is a singular Riemannian foliation on $\sphere ^n$ there is a natural
extension of $\mathcal F$ to a singular Riemannian foliation  $C\mathcal F$ on $\RR ^{n+1}$.  The  leaves  of the cone $C\mathcal F$ of $\mathcal F$ are the images of  leaves of $\mathcal F$ under the natural dilations $x\to r\cdot x$, for $r\in [0,\infty )$.

\section{Smoothness of the averaging operator}
\subsection{Measurable properties}
Let $(M,\fol)$ be a singular Riemannian foliation with compact leaves.  Let $M_0$ be the principal stratum as above.
 Since $M\setminus M_0$ has measure $0$ we can restrict ourselves to $M_0$ in all question which concern only almost everywhere properties of functions, in particular, when dealing with integrable and square-integrable functions. The subsequent considerations can be found in a much more general situation in \cite{PR}, thus we only sketch the arguments.

 Applying Fubini's theorem  to the Riemannian submersion
$\pi:M_0\to B_0$ we see that for any locally  integrable function $f\in L^1 _{loc} (M)$, the restriction
of $f$ to almost any fiber of $\pi$  (i.e. a leaf of $\mathcal F$) is integrable.  Moreover, for any compact  subset of $K$ of $B_0$   we have the equality
 $$\int _{\pi^{-1} (K)}  f =\int _K \left(\int _ {\pi^{-1} (q)} f\right)  d\vol_B(q)$$

Thus the averaging map $f\to [f]$ is well defined for any $f\in L^1_{loc} (M)$.
We identify $L^2(M)$ with $L^2(M_0)$. From the above  formula and the inequality of Cauchy-Schwarz  we deduce (cf. \cite{PR}) that for any
$f\in L^2 (M_0)$ the average $[f]$ defined by  Equation \ref{main} is indeed an element
of $L^2 (M_0)$. Moreover, we see that the averaging operator has norm $1$, hence
$\int _M f^2 \geq \int _M [f] ^2$.
By definition, the averaging operator is linear. For any function $f\in L^2 (M_0) =L^2(M)$
the average function $[f]$ is constant on almost all leaves, hence it has a basic representative.  On the other hand, if $f$ is constant on almost all leaves then $[f] =f$ in
$L^2 (M_0)$.  The kernel of the averaging operator consists of  all functions  $g \in L^2(M_0)$ whose average on almost any leaf is zero. In this case, for any basic function $ f \in L^2(M_0)$   the product
$f\cdot g$ still has average $0$ on almost all leaves and therefore $\int _M f\cdot g =0$.
Hence the kernel of the  averaging operator $[\cdot ]$ is orthogonal to its image. This shows that $[\cdot ]$
is indeed the orthogonal projection from $L^2 (M)$ onto the subset $L^2(M,\mathcal F)$ of all functions in $L^2 (M)$ which have a basic representative.

\subsection{Commuting operators}
Let us now assume that $\mathcal F$ has basic mean curvature  $H$. Denote as above by
$\kappa$ the  corresponding basic  smooth  $1$-form on $M_0$.    Note that for any smooth function $f:M_0 \to \RR$ the average function
$[f]:M_0 \to \RR$ is a smooth basic function.  The smoothness is evident, since in $M_0$
the leaves depend smoothly on the point. Indeed, the smoothness statement  is a trivial
case of \cite{PR}.  Using our assumption on the mean curvature we are going to conclude that the average operator (on the principal part  $M_0$) commutes with basic horizontal derivatives and with the Laplacian.

First we claim:
\begin{lem} \label{firstlem}
Let $X$ be a smooth, basic horizontal  vector field on $M_0$ and let $f$ be a smooth function. Then $[X(f)] =X([f])$.
\end{lem}

\begin{proof}
Both sides are linear in $f$ and clearly agree on smooth basic functions.  Thus
it suffices to prove the equality for all smooth functions $f$ with $[f]=0$, since
any function $f$ is the sum of $[f]$ and  $f-[f]$.

Denote by $\omega$ the volume form  of the leaves, well defined up to a sign.  The  mean curvature describes the infinitesimal volume change along the flow of $X$, hence
the Lie derivative of the  measures $\omega$ along the vector field $X$ is given by  $L_X (\omega) =-\kappa (X)\cdot \omega$ (\cite[Proposition 4.1.1]{GW}).  Since $\kappa$ is basic, the function $\kappa (X)$ is constant along each leaf.
Thus for any function $f$ with $[f]\equiv 0$ and any $p\in M_0$   we have
$$0=X(0) = X\Big(\int _{L_p}  f \omega \Big) = \int _{L_p} X(f) \omega + \int _{L_p} f L_X (\omega ) =\int _{L_p} X(f) \omega - \kappa (X) (p) \int _{L_p} f \omega$$
The last summand vanishes by assumption, hence $\int _{L_p}   X(f) = 0$.
The above equation implies $[X(f)]=0= X[f]$.
\end{proof}

From  the previous Lemma we are going to deduce  that the averaging operator commutes with the Laplacian (cf. \cite{PR}, Propositions 4.1 and 4.3).

\begin{lem} \label{secondlem}
If $\Delta$ denotes the Laplacian on $M_0$ then $\Delta [f] =[\Delta f]$ for any smooth function $f:M_0\to \RR$.
\end{lem}

\begin{proof}
Fix  a  point $p\in M_0$, consider an orthonormal frame $\{X_1, \ldots X_{k}, V_1,\ldots V_{n-k}\}$ in a neighborhood of $p$ where the $X_i$ are basic and the $V_i$ are vertical.

 Define the basic, resp. vertical Laplacians
$\Delta^h, \Delta^v$ as
\[
\Delta^hf=\sum_{i=1}^k\left(X_iX_i(f)- \nabla_{X_i}X_i(f)\right)\qquad \Delta^vf=\sum_{i=1}^{n-k}\left(V_iV_i(f)- \nabla_{V_i}V_i(f)\right)
\]
 The operators $\Delta^v$, $\Delta^h$ do not depend on the choice of the vertical and horizontal frames, and moreover $\Delta=\Delta^h+\Delta^v$ is the usual Laplacian.
Hence it suffices to prove the following identities:
\begin{align}
\label{E:lap-h}\Delta^h[f]&=[\Delta^hf]\\
\label{E:lap-v}\Delta^v[f]&=[\Delta^vf].
\end{align}

Since  the O'Neill tensor is skew-symmetric we see $\Delta^hf=\sum_iX_iX_i(f)-\nabla_{X_i}^hX_i(f)$.  Hence  $\Delta ^h$ is a sum of compositions of derivatives along basic horizontal fields. But the averaging operator $[\cdot ]$ commutes with derivations along basic horizontal fields by Lemma \ref{firstlem}.  This implies  Equation \eqref{E:lap-h}.

On the other hand, consider the operator $\Delta ^l (f):= \sum_{i=1}^{n-k}\left(V_iV_i(f)- \nabla_{V_i}^v V_i(f)\right) $  which is just the Laplacian
along  the leaves of the restriction of $f$ to the leaves.  By definition
$$\Delta ^l(f)= \Delta ^v (f) + \sum_{i=1}^{n-k} \nabla_{V_i}^h V_i(f) =
\Delta ^v (f) +H(f)$$
Due to Lemma \ref{firstlem}, the derivation along the basic field $H$ commutes with
$[\cdot ]$.   Moreover, since the Laplacian of a constant function is $0$ and since the integral of the Laplacian of any function on any compact manifold is $0$, we get for any smooth function $f:M_0\to \RR$
$$\Delta ^l  [f]\equiv 0\equiv [\Delta ^l f]$$
In particular, $\Delta ^l$ commutes with the averaging operator as well.
This implies Equation \eqref{E:lap-v}.
\end{proof}

\subsection{Boot-strapping to smoothness}
Under the assumptions above we are going to prove now that for any smooth function $f:M\to \RR$, the smooth average function $[f] :M_0 \to \RR$
has a smooth extension to $M$.
\begin{thm} \label{C:smooth}
Let $\mathcal F$ be a singular Riemannian foliation with compact leaves on a Riemannian manifold $M$.  Assume that  $\mathcal F$ has basic mean curvature. Then for any smooth function $f:M\to \RR$ the average function $[f] \in L^1_{loc} (M)$  has a
smooth representative.
\end{thm}

\begin{proof}
For any smooth function $f:M \to \RR$, denote by $[f] :M_0 \to \RR$ the smooth representative of the averaged function defined at every point of $M_0$ by \eqref{main}.  We claim that $[f]$ has a smooth extension  to $M_1$, the union of all strata of codimension at most $1$. (This is again a special case of \cite{PR}).
Indeed, either $M_1=M_0$, or the lift $\mathcal F'$ of $\mathcal F$ to a double cover $M_1'$ of $M_1$ has only principal leaves.  Since the averaging in
$M_1'$ with respect to $\mathcal F'$  commutes with the deck transformations of the cover $M_1' \to M_1$, we obtain the smoothness of the lift of $[f]$ to $M_1'$ and therefore the smoothness of $[f]$ on $M_1$.
Note, that since $\mathcal F$ is a regular foliation on $M_1$ the mean curvature field $H$ extends to a smooth vector field on $M_1$.  Moreover,
Lemma \ref{firstlem} and Lemma \ref{secondlem} remain true for smooth functions on $M_1$.

Next, we  claim that $[f]$ has a locally Lipschitz extension to $M$.  Consider
an arbitrary  point $p\in M$.  It is enough to find  a neighborhood $U$ of
$p$ in $M$ such that $[f]:U\cap M_1 \to \RR$ is Lipschitz continuous.   In order to do so, consider
 an open  $\mathcal F$-saturated pre-compact neighborhood $V$ of
$p$ in  $M$.  Since $f$ is smooth and  $\bar V$ compact, $f$ must be  $K$-Lipschitz on $V$ for some $K>0$.
Hence, for any unit basic horizontal vector field $X$ on $V_1:=V\cap M_1$ we have $|X(f)| \leq K$.
Due to Lemma \ref{firstlem}, we deduce that  $\big|X([f])\big| \leq K$ on $V_1$.  Since $[f]$ is basic, its gradient field is basic as well and we deduce that
$[f]:V_1\to \RR$ is \emph{locally} $K$-Lipschitz.  Consider now a small convex  ball $U\subset V$ around $p$ in $M$. Since $M\setminus M_1$ has codimension at least $2$ in $M$, any pair  of points in
$U_1:=U \cap M_1$ can be connected in $U_1$ by a smooth curve of length arbitrary close to the distance between these points.  Integrating along this curve  we deduce that $[f]:U_1 \to \RR$ is $K$-Lipschitz continuous. This finishes the proof of the claim.

The function $\Delta f$ is smooth as well as $f$. Due to the previous claim, the functions  $[f],[\Delta f]$ are both locally Lipschitz in $M$.
 If we proved that $[f]\in \mathcal C^2$ for any smooth function $f$ and that $\Delta[f]=[\Delta f]$ on the whole of $M$, then a standard bootstrap argument would prove the smoothness of $[f]$. Since the complement $Y=M\setminus M_1$ of the regular stratum has codimension $\geq 2$ in $M$, the following analytic  Proposition \ref{ana} together with Lemma \ref{secondlem} provides exactly what we need, thus finishing the proof of the Theorem.
\end{proof}

\begin{prop} \label{ana}
Let $M$ be an $n$-dimensional  Riemannian manifold and let $Y$ be a
closed subset  of $M$ with vanishing $(n-1)$-dimensional Hausdorff measure.
Assume that $u,g$ are locally Lipschitz functions on $M$ such that  $\Delta u =g$
on $M\setminus Y$ in the sense of distributions.  Then the function $u$
is of class $\mathcal C^2$. Moreover, $\Delta u=g$ on  $M$.
\end{prop}

\begin{proof}
The question is local and we  may restrict to a small open ball $B$ around a
given point in $M$. In this ball we can solve the Dirichlet problem and find a
map $u_1$ in the Sobolev class $H^{1,2} (B)$, with $\Delta u_1=g$ in $B$.  Since $g$ is Lipschitz continuous,  elliptic regularity gives us $u_1 \in \mathcal C^2  (B)$. If  we can  prove that the locally Lipschitz function $u_2=u-u_1$ is harmonic in $B$, then the regularity of $u$ would follow from the regularity of $u_1$ and $u_2$.

Since $u$ and $u_1$ are in the Sobolev space $H^{1,2} (B)$, so is their difference $u_2$.
Hence the  Laplacian $\Delta u_2$ is a distribution in the Sobolev space $H^{-1,2} (B)$. By assumption this distribution has its support on $Y\cap B$. Since $Y$ has vanishing $(n-1)$-dimensional Hausdorff measure, it follows from \cite[p. 16]{Mazya} and   \cite[p. 70]{Adams} that the only distribution in $H^{-1,2} (B)$
with support in $Y$ is $0$.  Therefore, $\Delta u_2=0$ on $B$.
\end{proof}

\section{Homogeneous basic polynomials}

In this section, we consider a singular Riemannian foliation  $\mathcal F$ with compact leaves on a round sphere $\sphere^n$. Consider the induced foliation $C\mathcal F$ on the Euclidean space $V=\RR ^{n+1}$ invariant under canonical dilations.
 By \cite{AR15}, both foliations have basic mean curvature and the results from the previous section, show that for a smooth function $f:V \to \RR$ its averaged function $[f]$ is smooth as well. Since the leaves of $C\mathcal F$ through points in $\sphere ^n $ coincide with the corresponding leaves of $\mathcal F$, the average of
$f|_{\sphere ^n}$ with respect to $\mathcal F$ is just the restriction of
$[f]$ to the sphere  ${\sphere ^n}$.

The following observation together with Theorem \ref{C:smooth} finishes the proof of Theorem \ref{smooth average}:
\begin{prop}
If $f:V \to \RR$ is a homogeneous polynomial, $[f]$ is a homogeneous polynomial of the same degree.
\end{prop}
\begin{proof}
A smooth function $f:V \to \RR$ is a homogeneous polynomial of degree $m$ if and only if
\begin{equation} \label{homo}
f(rx)= r^m f(x)
\end{equation} holds true for all $x\in V$ and $r\in [0,\infty )$, as one can see from the Taylor expansion.
   Since the foliation $C\mathcal F$ is invariant under dilations,  equality \eqref{homo} for the function $f$ implies the  same equality for the average function $[f]$.  Thus the result follows from Theorem \ref{C:smooth}.
\end{proof}

Consider now the ring $\RR[V]^b$ of basic polynomials on $V$ with respect to $C\mathcal F$.  This is a subring of the ring $\RR[V] =\RR [x_1,....,x_{n+1}]$.
Since the average  $[\cdot ]:\RR[V]\to \RR[V]^b$ preserves the degree,  we see that
$\RR[V]^b$  is homogeneous:  For any polynomial $p\in \RR[V]^b$, the homogeneous  summands of $p$  are again in $\RR[V]^b$.

 Hilbert's proof of finite generation of the rings of invariants (cf. \cite[p. 274]{We}) applies to our situation:
\begin{prop}  \label{hilbert}
The ring $\RR[V]^b$ of basic polynomials is finitely generated.
\end{prop}
\begin{proof}
By Hilbert's Basis Theorem, the ideal $I$ in $\RR[V]$  generated by the subring $\RR[V]^b_+$ of basic polynomials of positive degree,     is finitely generated   (as a module over $\RR[V]$).   Thus we can find  homogeneous basic
polynomials
 $\rho_1,\ldots \rho_k$  of positive degrees which generate $I$ as an ideal.

We now prove that  $\RR[V]^b=\RR[\rho_1,\ldots \rho_k]$ as a ring, proceeding by induction on the degree. Assume that all $q\in \RR[V]^b$ of degree smaller than $m$ are contained in $\RR[\rho_1\ldots \rho_k]$, and consider some homogeneous $p\in \RR[V]^b$ of degree $m$.
Since $p\in \RR[V]^b_+\subset I$, we can find polynomials $a_1, \ldots a_k \in \RR[V]$ such that
\[
p=\sum a_i \rho_i.
\]
Moreover, we may assume that  each $a_i$ is homogeneous of degree smaller than $m$.
We apply our averaging operator to this equation  and obtain
\[
p=\sum [a_i]\rho_i.
\]
By induction, the basic polynomials $[a_i]$  are contained in $\RR[\rho_1,\ldots\rho_k]$.  Therefore,
$p$ is contained in $\RR[\rho_1,\ldots \rho_k]$ as well.
\end{proof}

There are plenty of basic polynomials:

\begin{prop}
The ring of basic polynomials separates different leaves of  $C\mathcal F$.
\end{prop}

\begin{proof}
Given two leaves $L_x$, $L_y$, consider a smooth function $f$ (a bump function) which is constant  1 in $L_y$ and constant $0$ in $L_x$. By the theorem of Weierstrass, there exists a polynomial  $P:\RR^{n+1} \to \RR$  such that $|f-P|<\epsilon$ on the compact set $L_x\cup L_y$. Then $[P](x)\in (-\epsilon,\epsilon)$, $[P](y)\in (1-\epsilon,1+\epsilon)$, therefore, $[P]$ separates  $L_x, L_y$.
\end{proof}

The following result  finishes the proof of Theorem \ref{T:algebraic}:

\begin{prop}
In the notations above, let $\rho_1,...,\rho _k$ be generators of the ring $ \RR[V]^b$ of
basic polynomials.
The map $\rho
=(\rho_1,\ldots \rho_k):\RR^{n+1}\to \RR^k$ descends to a homeomorphism of $\RR^{n+1}/\fol$ onto its image.
\end{prop}
\begin{proof}
Since any coordinate $\rho _i$ of $\rho$ is basic, the map $\rho$ descends to a map $\rho_*:\RR^{n+1}/\fol \to \RR^k$.
Since the basic polynomials separate points and $\rho_i$ generate the ring of all basic polynomials, the map $\rho_* : \RR^{n+1}/\fol \to \RR^k$ is a bijection onto its image.  In particular, the non-empty fibers of $\rho$ coincide with leaves of $C\mathcal F$, thus $\rho$ is proper and so is $\rho_*$.
Therefore, the map
$\rho:\RR^{n+1}/\fol \to \RR ^k$  is a homeomorphism onto the image.
\end{proof}

-----------------------------------------------------

\bibliographystyle{amsplain}

\end{document}